\documentclass{scrartcl}

\usepackage{amsmath}
\usepackage{amsthm}
\usepackage[thinlines]{easytable}
\usepackage{enumerate}
\usepackage{thm-restate}
\usepackage{xcolor}

\usepackage[colorlinks=true,pagebackref=true]{hyperref}
\hypersetup{
	linkcolor = {blue},
	citecolor = {blue}
}
\usepackage{cleveref}

\newcommand{\fO}{\mathcal{O}}
\newcommand{\fP}{\mathcal{P}}

\newcommand{\rt}{\mathrm{T}}
\newcommand{\rb}{\mathrm{B}}

\newcommand{\rH}{\mathrm{H}}
\newcommand{\rV}{\mathrm{V}}

\newcommand{\ex}{\mathrm{ex}}
\newcommand{\sat}{\mathrm{sat}}

\newcommand{\zeromat}{\mathbf{0}}

\newcommand{\Iprimetwomat}{\begin{smallbulletmatrix}&\o\\\o&\end{smallbulletmatrix}}

\newenvironment{smallbulletmatrix}{%
	\renewcommand{\o}{\bullet}
	\renewcommand{\d}{\cdot}
	\renewcommand{\r}{\textcolor{red}{\o}}
	\renewcommand{\b}{\textcolor{blue}{\o}}

	\left(\begin{smallmatrix}%
}{%
	\end{smallmatrix}\right)%
}

\newtheorem{theorem}{Theorem}[section]
\newtheorem{lemma}[theorem]{Lemma}
\newtheorem{proposition}[theorem]{Proposition}

\newtheorem*{openQuestion}{Open Question}

\title{Matrix patterns with bounded saturation function}
\author{Benjamin Aram Berendsohn\thanks{Institut f\"ur Informatik, Freie Universit\"at Berlin, \texttt{beab@zedat.fu-berlin.de}. Work supported by DFG grant KO 6140/1-1.}}

\begin{document}
	
\maketitle

\begin{abstract}
	A 0-1 matrix $M$ \emph{contains} a 0-1 matrix \emph{pattern} $P$ if we can obtain $P$ from $M$ by deleting rows and/or columns and turning arbitrary 1-entries into 0s.
	The saturation function $\sat(P,n)$ for a 0-1 matrix pattern $P$ indicates the minimum number of 1s in a $n \times n$ 0-1 matrix that does not contain $P$, but changing any 0-entry into a 1-entry creates an occurrence of $P$. Fulek and Keszegh recently showed that the saturation function is either bounded or in $\Theta(n)$. Building on their results, we find a large class of patterns with bounded saturation function, including both infinitely many permutation matrices and infinitely many non-permutation matrices.
\end{abstract}

\section{Introduction}

In this paper, all matrices are  0-1 matrices. For a cleaner presentation, we write matrices with dots ($\begin{smallmatrix}\bullet\end{smallmatrix}$) instead of 1s and spaces instead of 0s, for example:
\begin{align*}
\left(
\begin{smallmatrix}
	0&1&0\\
	0&0&1\\
	1&0&0
\end{smallmatrix}\right)
=
\begin{smallbulletmatrix}
	  &\o&  \\
	  &  &\o\\
	\o&  &  
\end{smallbulletmatrix}
\end{align*}
In line with this notation, we call a row or column \emph{empty} if it only contains 0s. Furthermore, we refer to changing an entry from 0 to 1 as \emph{adding} a 1-entry, and to the reverse as \emph{removing} a 1-entry.

We index matrices as follows: The entry $(i,j)$ is in the $i$-th row (from top to bottom) and the $j$-th column (from left to right). For example, the above matrix has 1-entries $(1,2)$, $(2,3)$ and $(3,1)$.

A \emph{pattern} is a matrix that is not all-zero. A matrix $M$ \emph{contains} a pattern $P$ if we can obtain $P$ from $M$ by deleting rows and/or columns, and turning arbitrary 1-entries into 0s. If $M$ does not contain $P$, we say $M$ \emph{avoids} $P$.
Matrix pattern avoidance can be seen as a generalization of two other areas in extremal combinatorics: Pattern avoidance in permutations (see, e.g., Vatter's survey \cite{Vatter2014}), which corresponds to the case where both $M$ and $P$ are permutation matrices; and forbidden subgraphs in bipartite graphs, which corresponds to avoiding a pattern $P$ and all other patterns obtained from $P$ by permutation of rows and/or columns.\footnote{For this, we interpret the $M$ and $P$ as adjacency matrices of bipartite graphs.}

A classical question in extremal graph theory is to determine the minimum number of edges in a $n$-vertex graph avoiding a fixed pattern graph $H$. The corresponding problem in forbidden submatrix theory is determining the maximum \emph{weight} (number of 1s) of a $m \times n$ matrix avoiding the pattern $P$, denoted by $\ex(P, m, n)$. We call $\ex(P,n) = \ex(P,n,n)$ the \emph{extremal function} of the pattern $P$. The study of the extremal function originates in its applications to (computational) geometry \cite{Mitchell1987,Fueredi1990,BienstockGyoeri1991}. A systematic study initiated by Füredi and Hajnal \cite{FuerediHajnal1992} has produced numerous results \cite{Klazar2000,Klazar2001,MarcusTardos2004,Tardos2005,Keszegh2009,Fulek2009,Geneson2009,Pettie2011,Pettie2011a} and further applications in the analysis of algorithms have been discovered \cite{Pettie2010,ChalermsookEtAl2015}.


Clearly, for non-trivial patterns, $\ex(P,n)$ is at least linear and at most quadratic. Large classes of patterns with linear and quasi-linear extremal functions have been identified \cite{Keszegh2009,Pettie2011}. On the other hand, there are patterns with nearly quadratic extremal functions~\cite{AlonEtAl1999}.

A natural counterpart to the extremal problem is the \emph{saturation problem}. A matrix is \emph{saturated} for a pattern $P$ if it avoids $P$ and is maximal in this respect, i.e., turning any 0-entry of $M$ into a 1 creates an occurrence of $P$. Clearly, $\ex(P,m,n)$ can also be defined as the maximum weight of a $m \times n$ matrix that is saturated for $P$. The function $\sat(P,m,n)$ indicates the \emph{minimum} weight of a $m \times n$ matrix that is saturating for $P$. We focus on square matrices and the \emph{saturation function} $\sat(P,n) = \sat(P,n,n)$.

The saturation problem for matrix patterns was first considered by Brualdi and Cao \cite{BrualdiCao2020} as a counterpart of saturation problems in graph theory. Fulek and Keszegh \cite{FulekKeszegh2020} started a systematic study. They proved that, perhaps surprisingly, every pattern $P$ satisfies $\sat(P,n) \in \fO(1)$ or $\sat(P,n) \in \Theta(n)$. This is in stark contrast to the extremal problem. Further, they present large classes of patterns with linear saturation functions, and a single non-trivial pattern with bounded saturation function.

Most interesting for our purposes is a class of patterns we call \emph{once-separable}: A pattern is once-separable if it has the form
\begin{align*}
	\begin{pmatrix} A & \zeromat \\ \zeromat & B \end{pmatrix} \text{ or } \begin{pmatrix} \zeromat & A \\ B & \zeromat \end{pmatrix}
\end{align*}
for two patterns $A$ and $B$, where $\zeromat$ denotes an all-0 matrix of arbitrary dimensions.
\begin{theorem}[{\cite[Theorem 1.7]{FulekKeszegh2020}}]\label{p:linsat-once-sep}
	If $P$ is once-separable, then $\sat(P,n) \in \Theta(n)$.
\end{theorem}

In this paper, for the sake of simplicity, we only consider patterns with no empty rows or columns. However, we note that the saturation function, unlike the extremal function, may change considerably by the addition of an empty row or column. In particular, Fulek and Keszegh proved that if the first or last row or column of a pattern $P$ is empty, then $\sat(P,n) \in \Theta(n)$.

Note that if $P'$ can be obtained from $P$ by rotation, inversion\footnote{Swapping the role of rows and columns.}, or reflection\footnote{Reversing all rows or all columns.}, then $\sat(P,n) = \sat(P',n)$.

\paragraph{Permutation matrix patterns.}

In this paper, we give special attention to \emph{permutation matrix} patterns. A permutation matrix is a square matrix where every row and every column contains exactly one 1-entry. \Cref{p:linsat-once-sep} already covers the once-separable permutation matrices.

We call the 1-entries in the first or last row or column the \emph{outer} 1-entries. It is easy to see that a not-once-separable permutation matrix cannot have a 1-entry in one of its corners. As such, up to reflection, the outer four 1-entries form one of the patterns $Q_0$ and $Q_1$, where
\begin{align*}
	Q_0 = \begin{smallbulletmatrix}
		  &\o&  &  \\
		\o&  &  &  \\
		  &  &  &\o\\
		  &  &\o&  
	\end{smallbulletmatrix},\hspace{10mm}
	Q_1 = \begin{smallbulletmatrix}
		  &  &\o&  \\
		\o&  &  &  \\
		  &  &  &\o\\
		  &\o&  &  
	\end{smallbulletmatrix}.
\end{align*}

In particular, all $3 \times 3$ permutation matrices are once-separable, and $Q_1$ is the only $4 \times 4$ pattern that is not once-separable. The $5 \times 5$ not-once-separable matrices are shown in \Cref{fig:small-perm-pats}. Fulek and Keszegh already proved that $Q_2$ has bounded saturation function, and ask whether the same is true for $Q_1$.

Call a permutation matrix \emph{$Q_0$-like} (\emph{$Q_1$-like}) if the outer 1-entries form $Q_0$ (respectively, $Q_1$). We prove that all $Q_1$-like permutation matrices have bounded saturation function.
\begin{theorem}\label{p:q1-like-perm}
	Let $P$ be a $Q_1$-like $k \times k$ permutation matrix. Then $\sat(P,n) \in \fO(1)$.
\end{theorem}

This covers the pattern $Q_1$ (thus answering the question of Fulek and Keszegh) and the patterns $Q_2$, $Q_3$, and $Q_4$ in \Cref{fig:small-perm-pats}. For permutation matrices of size at most 6, we obtain a full characterization of the saturation functions with the following theorem.

\begin{restatable}{theorem}{restateSmallPatterns}\label{p:small-patterns}
	Let $P$ be a not-once-separable $k \times k$ permutation matrix with $k \le 6$. Then $\sat(P,n) \in \fO(1)$.
\end{restatable}


\begin{figure}[tbp]
	\begin{align*}
		Q_2 = \begin{smallbulletmatrix}
			  &  &  &\o&  \\
			\o&  &  &  &  \\
			  &  &\o&  &  \\
			  &  &  &  &\o\\
			  &\o&  &  &  
		\end{smallbulletmatrix},
		Q_3 = \begin{smallbulletmatrix}
			  &  &  &\o&  \\
			\o&  &  &  &  \\
			  &\o&  &  &  \\
			  &  &  &  &\o\\
			  &  &\o&  &  
		\end{smallbulletmatrix},
		Q_4 = \begin{smallbulletmatrix}
			  &  &  &\o&  \\
			\o&  &  &  &  \\
			  &  &  &  &\o\\
			  &\o&  &  &  \\
			  &  &\o&  &  
		\end{smallbulletmatrix},
		Q_5 = \begin{smallbulletmatrix}
			  &  &\o&  &  \\
			\o&  &  &  &  \\
			  &  &  &  &\o\\
			  &\o&  &  &  \\
			  &  &  &\o&  
		\end{smallbulletmatrix}
	\end{align*}
	\caption{$5 \times 5$ permutation matrices with bounded saturation function, up to rotation and reflection.}\label{fig:small-perm-pats}
\end{figure}

\paragraph{Other patterns.} We call a pattern \emph{non-trivial} if it has two rows that only have a 1 in the leftmost (respectively rightmost) position, and two columns which only have a 1 in the topmost (respectively bottommost) position. Otherwise, we call the pattern \emph{trivial}. Fulek and Keszegh show that each trivial pattern has linear saturation function~\cite[Theorem 1.11]{FulekKeszegh2020}. Note that every permutation matrix is non-trivial.

\begin{figure}[htbp]
	\begin{align*}
		\begin{smallbulletmatrix}
			  &  &\o&  \\
			\o&  &  &  \\
			  &  &  &\o\\
			\o&\o&  &  
		\end{smallbulletmatrix},\hspace{10mm}
		\begin{smallbulletmatrix}
			  &  &\o&  \\
			\o&  &  &  \\
			  &  &  &\o\\
			  &\o&\o&  
		\end{smallbulletmatrix}.
	\end{align*}
	\caption{A non-trivial pattern (left), and a trivial pattern (right).}
\end{figure}

Our techniques easily generalize to a more general class of non-trivial patterns (in fact, we only prove them in the general form). We restrict ourselves to the patterns without empty rows or columns where the first and last row and column each contain only a single 1-entry. Since the case of once-separable patterns is already solved, this again leaves us with patterns where the outer 1-entries form either $Q_0$ or $Q_1$ (up to reflection). We extend our previous definitions as follows: An arbitrary pattern is called \emph{$Q_0$-like} (\emph{$Q_1$-like}) if it has no empty rows and columns, and exactly four outer entries that form an occurrence of $Q_0$ (respectively, $Q_1$). We prove a generalization of \Cref{p:q1-like-perm}.

\begin{restatable}{theorem}{restateQOneLike}\label{p:q1-like}
	Let $P$ be a non-trivial $Q_1$-like $k \times k$ pattern. Then $\sat(P,n) \in \fO(1)$.
\end{restatable}

We prove \Cref{p:q1-like} (which implies \Cref{p:q1-like-perm}) and \Cref{p:small-patterns} in \Cref{sec:wit-constr}. All our results are based on the construction of a \emph{witness}, a concept introduced by Fulek and Keszegh. In \Cref{sec:wit}, we formalize and develop this notion, based on the proof by Fulek and Keszegh that $Q_2$ has bounded saturation function.

\section{Witnesses}\label{sec:wit}

Let $P$ be a matrix pattern without empty rows or columns. An \emph{explicit witness} (called simply \emph{witness} by Fulek and Keszegh~\cite{FulekKeszegh2020}) for $P$ is a matrix $M$ that is saturated for $P$ and contains at least one empty row and at least one empty column. Clearly, if $\sat(P,n) \in \fO(1)$, then $P$ has an explicit witness. Fulek and Keszegh note that the reverse is also true: We can replace an empty row (column) by an arbitrary number of empty rows (columns), and the resulting arbitrarily large matrix will still be saturating for $P$.\footnote{Note that it is critical here that $P$ has no empty rows or columns. Otherwise,  increasing the number of empty rows or columns in $M$ might create an occurrence of $P$.} As such, an $m_0 \times n_0$ explicit witness for $P$ of weight $w$ implies that $\sat(P,m,n) \le w$ for each $m \ge m_0$ and $n \ge n_0$.

We call a row (column) of a matrix $M$ \emph{expandable} w.r.t.\ $P$ if the row (column) is empty and adding a single 1-entry anywhere in that row (column) creates a new occurrence of $P$ in $M$. A explicit witness for $P$ is thus a saturated matrix with at least one expandable row and an expandable column w.r.t.\ $P$. We define a \emph{witness} for $P$ (used implicitly by Fulek and Keszegh) as a matrix that avoids $P$ and has at least one expandable row and at least one expandable column w.r.t.\ $P$. Clearly, an explicit witness is a witness. The following lemma shows that finding a (general) witness is sufficient to show that $\sat(P,n) \in \fO(1)$.

\begin{lemma}\label{p:gwitness}
	If a pattern $P$ without empty rows or columns has a $m_0 \times n_0$ witness, then $P$ has a $m_0 \times n_0$ explicit witness.
\end{lemma}
\begin{proof}
	Let $M$ be a $m_0 \times n_0$ witness for $P$. If $M$ is saturating for $P$, then we are done. Otherwise, there must be a 0-entry $(i,j)$ in $M$ that can be changed to 1 without creating an occurrence $P$. Note that $(i,j)$ cannot be contained in an expandable row or column of $M$, so the resulting matrix is still a witness. Thus, we obtain an explicit witness after repeating this step at most $m_0 \cdot n_0$ times.
\end{proof}

\subsection{Saturating matrices with constant width or height}

Fulek and Keszegh also considered the asymptotic behavior of the functions $\sat(P,m_0,n)$ and $\sat(P,m,n_0)$, where $m_0$ and $n_0$ are fixed. The dichotomy of $\sat(P,n)$ also holds in this setting:

\begin{theorem}[{\cite[Parts of Theorem 1.3]{FulekKeszegh2020}}]\label{p:dichotomy}
	For every pattern $P$, and constants $m_0, n_0$,
	\begin{enumerate}[(i)]
		\itemsep0pt
		\item either $\sat(P,m_0,n) \in \fO(1)$ or $\sat(P,m_0,n) \in \Theta(n)$;\label{item:dich_hor}
		\item either $\sat(P,m,n_0) \in \fO(1)$ or $\sat(P,m,n_0) \in \Theta(m)$.\label{item:dich_vert}
	\end{enumerate}
\end{theorem}

We can adapt the notion of witnesses in order to classify $\sat(P,m_0, n)$ and $\sat(P,m,n_0)$. Let $P$ be a matrix pattern without empty rows or columns. A \emph{horizontal (vertical) witness} for $P$ is a matrix $M$ that avoids $P$ and contains an expandable column (row).\footnote{A horizontal witness can be expanded horizontally, a vertical witness can be expanded vertically.} Clearly, $P$ has a horizontal witness with $m_0$ rows if and only if $\sat(P, m_0, n)$ is bounded; and $P$ has a vertical witness with $n_0$ columns if and only if $\sat(P,m,n_0)$ is bounded. Further note that $M$ is a witness for $P$ if and only if $M$ is horizontal witness and a vertical witness.

Observe that rotation and inversion of $P$ may affect the functions $\sat(P,m,n_0)$ or $\sat(P,m_0,n)$, but reflection does not.

\begin{lemma}\label{p:ext-hor-wit}
	Let $P$ be a matrix pattern without empty rows or columns, and only one entry in the last row (column). Let $W$ be a horizontal (vertical) witness for $P$. Then, appending an empty row (column) to $W$ again yields a horizontal (vertical) witness.
\end{lemma}
\begin{proof}
	We prove the lemma for horizontal witnesses, and appending a row. The other case follows by symmetry.
	Let $W$ be a $m_0 \times n_0$ horizontal witness for $P$, where the $j$-th column of $W$ is expandable. Let $W'$ be a matrix obtained by appending a row. Clearly, $W'$ still does not contain $P$. Moreover, adding an entry in $W'$ at $(i,j)$ for any $i \neq n_0+1$ creates a new occurrence of $P$. It remains to show that adding an entry at $(n_0+1,j)$ creates an occurrence of $P$.
	
	We know that adding an entry at $(n_0,j)$ in $W'$ creates an occurrence of $P$. Let $I$ the set of positions of 1-entries in $W(P)$ that form the occurrence of $P$. Since $P$ has only one entry in the last row, all positions $(i', j') \in I \setminus \{(n_0,j)\}$ satisfy $i' < n_0+1$. Thus, adding a 1-entry at $(n_0+1,j)$ instead of $(n_0,j)$ creates an ocurrence of $P$ at positions $I \setminus \{(n_0,j)\} \cup \{(n_0+1,j)\}$, which implies that $W'$ is a horizontal witness.
\end{proof}

We now prove the following handy lemma, that allows us restrict our attention to the classification of $\sat(P,m_0, n)$ and $\sat(P,m,n_0)$. It essentially is a generalization of the technique used by Fulek and Keszegh to prove that $\sat(Q_2,n) \in \fO(1)$.

\begin{lemma}\label{p:vert_hor_wit}
	Let $P$ be a not-once-separable pattern without empty rows or columns, and with only one 1-entry in the last row and one 1-entry in the last column. Then $\sat(P,n) \in \fO(1)$ if and only if there exist constants $m_0, n_0$ such that $\sat(P,m_0, n) \in \fO(1)$ and $\sat(P,m, n_0) \in \fO(1)$.
\end{lemma}
\begin{proof}
	Suppose that $\sat(P,n) \in \fO(1)$. Then $P$ has a $m_0 \times n_0$ witness $M$, and thus $\sat(P,m_0, n)$ is at most the weight of $M$, for every $n \ge n_0$. Similarly, $\sat(P,m, n_0) \in \fO(1)$.
	
	Now suppose that $\sat(P,m_0, n) \in \fO(1)$ and $\sat(P,m, n_0) \in \fO(1)$. Then, for some $m_1, n_1$, there exists a $m_0 \times n_1$ horizontal witness $W_\rH$ and a $m_1 \times n_0$ vertical witness $W_\rV$. Consider the following $(m_0+m_1) \times (n_0+n_1)$ matrix, where $\zeromat_{m \times n}$ denotes the all-0 $m \times n$ matrix:
	\begin{align*}
		W = \begin{pmatrix} \zeromat_{m_0 \times n_0} & W_\rH \\ W_\rV & \zeromat_{m_1 \times n_1} \end{pmatrix}
	\end{align*}
	
	We first show that $W$ does not contain $P$. Suppose it does. Since $P$ is contained neither in $W_\rH$ nor in $W_\rV$, an occurrence of $P$ in $W$ must contain 1-entries in both the bottom left and top right quadrant. But then $P$ must be once-separable, a contradiction.
	
	By \Cref{p:ext-hor-wit}, $W_\rV' = (W_\rV, \zeromat_{m_1 \times n_1})$ is a vertical witness, and $W_\rH' = \binom{W_\rH}{\zeromat_{m_1 \times n_1}}$ is a horizontal witness. The expandable row in $W_\rV'$ and the expandable column in $W_\rH'$ are both also present in $W$. This implies that $W$ is a witness for $P$, so $\sat(P, n) \in \fO(1)$.
\end{proof}

\Cref{fig:wit-genwit} shows an example of a witness for $Q_1$, constructed with \Cref{p:vert_hor_wit}, using vertical/horizontal witnesses presented later in \Cref{sec:wit-constr}, and an explicit witness constructed using \Cref{p:gwitness}.

\begin{figure}
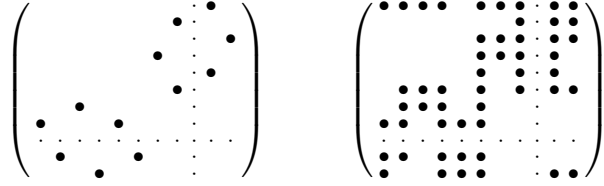

	\centering
	\begin{align*}
		\begin{smallbulletmatrix}
			  &  &  &  &  &  &  &  &\d&\o&  \\
			  &  &  &  &  &  &  &\o&\d&  &  \\
			  &  &  &  &  &  &  &  &\d&  &\o\\
			  &  &  &  &  &  &\o&  &\d&  &  \\
			  &  &  &  &  &  &  &  &\d&\o&  \\
			  &  &  &  &  &  &  &\o&\d&  &  \\
			  &  &\o&  &  &  &  &  &\d&  &  \\
			\o&  &  &  &\o&  &  &  &\d&  &  \\
			\d&\d&\d&\d&\d&\d&\d&\d&\d&\d&\d\\
			  &\o&  &  &  &\o&  &  &\d&  &  \\
			  &  &  &\o&  &  &  &  &\d&  &  
		\end{smallbulletmatrix}
		\hspace{10mm}
		\begin{smallbulletmatrix}
			\o&\o&\o&\o&  &\o&\o&\o&\d&\o&\o\\
			  &  &  &  &  &  &  &\o&\d&\o&\o\\
			  &  &  &  &  &\o&\o&\o&\d&\o&\o\\
			  &  &  &  &  &\o&\o&\o&\d&\o&  \\
			  &  &  &  &  &\o&  &\o&\d&\o&  \\
			  &\o&\o&\o&  &\o&  &\o&\d&\o&\o\\
			  &\o&\o&\o&  &\o&  &  &\d&  &  \\
			\o&\o&  &\o&\o&\o&  &  &\d&  &  \\
			\d&\d&\d&\d&\d&\d&\d&\d&\d&\d&\d\\
			\o&\o&  &\o&\o&\o&  &  &\d&  &  \\
			\o&  &  &\o&\o&\o&  &  &\d&\o&\o
		\end{smallbulletmatrix}
	\end{align*}
	\caption{A witness (left) and an explicit witness (right) for the pattern $Q_1$. The small dots indicate the expandable row/column.}\label{fig:wit-genwit}
\end{figure}

For certain classes of patterns closed under rotation or inversion, we can further restrict our attention to only vertical witnesses.

\begin{lemma}\label{p:vert_wit_suff}
	Let $\fP$ be a class of not-once-separable patterns without empty rows or columns, and with only one 1-entry in the last row and one 1-entry in the last column. If $\fP$ is closed under rotation or inversion and each pattern in $\fP$ has a vertical witness, then $\sat(P,n) \in \fO(1)$ for each $P \in \fP$.
\end{lemma}
\begin{proof}
	By \Cref{p:vert_hor_wit}, it suffices to show that each pattern in $\fP$ has a horizontal witness. Let $P \in \fP$ and let $P' \in \fP$ be obtained by rotating $P$ by 90 degrees clockwise (respectively, by inverting $P$). Let $W'$ be a vertical witness for $P'$, and let $W$ be obtained by rotating $W'$ by 90 degrees counterclockwise (respectively, by inverting $W'$). Clearly, $W$ is a horizontal witness for $P$. \Cref{p:vert_hor_wit} concludes the proof.
\end{proof}

\section{A simple witness construction}\label{sec:wit-constr}

We present a construction that yields vertical witnesses for certain non-trivial matrices with only one 1-entry in the first and last column. \Cref{fig:WP} shows an example of the construction. The idea is simple: Make two copies $P_1$ and $P_2$ of $P$, and arrange them in a way that the rightmost 1-entry of $P_1$ coincides with the leftmost 1-entry of $P_2$ (increase the matrix size as necessary, without creating empty rows or columns). Then, delete the column where $P_1$ and $P_2$ intersect. Note that this creates an empty row, which formerly contained the intersection of $P_1$ and $P_2$. Adding a 1-entry in that row creates an occurrence of $P$ by either completing $P_1$ or $P_2$, so that row is expandable. If now the constructed matrix also avoids $P$ (which is not necessarily the case), then it is a vertical witness for $P$. We now proceed with the formal definition and proof.

\begin{figure}
	\centering
	\begin{align*}
		\begin{smallbulletmatrix}
			  &  &\o&  \\
			\o&  &  &  \\
			  &  &  &\o\\
			  &\o&  &  
		\end{smallbulletmatrix}
		\rightarrow
		\begin{smallbulletmatrix}
			  &  &\o&  &  &  &  \\
			\o&  &  &  &  &\o&  \\
			  &  &  &\textcolor{red}{\o}&  &  &  \\
			  &\o&  &  &  &  &\o\\
			  &  &  &  &\o&  &  
		\end{smallbulletmatrix}
		\rightarrow
		\begin{smallbulletmatrix}
			  &  &\o&  &  &  \\
			\o&  &  &  &\o&  \\
			\d&\d&\d&\d&\d&\d\\
			  &\o&  &  &  &\o\\
			  &  &  &\o&  &  
		\end{smallbulletmatrix}
	\end{align*}
	\caption{Construction of $W(Q_1)$ from $Q_1$. The small dots indicate the expandable row.}\label{fig:WP}
\end{figure}

Let $P = (p_{i,j})_{i,j}$ be a $k \times k$ pattern with exactly one entry in the first column and exactly one entry in the last column. Let $s$ and $t$ be the rows of the leftmost and rightmost 1-entry in $P$, i.e., $p_{s,1} = 1$ and $p_{t,k} = 1$. Without loss of generality, assume that $s < t$. We define the $(k+t-s) \times (2k-2)$ matrix $W(P) = (w_{i,j})_{i,j}$ as follows:
\begin{align*}
	w_{i,j} = \begin{cases}
		p_{i,j}, & \text{ if } j < k, i \le k \\
		p_{i - (t-s), j - (k-2)}, & \text{ if } j \ge k, i \ge (t-s) + 1 \\
		0, & \text{ otherwise.}
	\end{cases}
\end{align*}

\begin{lemma}\label{p:wp_vert_wit}
	Let $P$ be a non-trivial pattern without empty rows and columns and with exactly one entry in the first and last column. If $W(P)$ avoids $P$, then $W(P)$ is a vertical witness for $P$.
\end{lemma}
\begin{proof}
	Since $P$ is non-trivial, the $s$-th and $t$-th rows of $P$ each only contain one entry, so the $t$-th row of $W(P)$ is empty. It remains to show that adding a 1-entry in the $t$-th row of $W(P)$ creates a new occurrence of $P$.
	
	Let $M$ be the matrix obtained by adding a 1-entry $(t,u)$ in $W(P)$. If $u \le k-1$, we remove the first $t-s$ rows and all columns other than the $u$-th and the last $k-1$. The result is $P$ with an additional 1-entry in the first column (which was the $u$-th column in $M$). If $u > k-1$, we remove the last $t-s$ rows and all columns except the $u$-th and the first $k-1$. The result is $P$ with an additional 1-entry in the last column.
\end{proof}

\subsection{Non-trivial \texorpdfstring{$Q_1$}{Q\_1}-like patterns}

\begin{lemma}\label{p:wp-q1-like}
	Let $P$ be a non-trivial $Q_1$-like pattern. Then $W(P)$ avoids $P$.
\end{lemma}
\begin{proof}
	Suppose $W(P)$ contains an occurrence of $P$, say at positions $I$. Consider the bottommost and topmost positions $(i_\rb, j_\rb), (i_\rt, j_\rt) \in I$. Since $P$ is $Q_1$-like, we have $j_\rb < j_\rt$. Moreover, $i_\rt - i_\rb \ge k-1$ (since $P$ has $k-2$ rows between the bottomost and topmost 1-entry).
	
	Consider first the case that $j_\rb \le k-1$. Then, by construction, $i_\rb \le k$, which implies that $I$ is completely contained in the first $k$ rows, including the empty $t$-th row. However, an occurrence of $P$ must have 1-entries in $k$ distinct rows, a contradiction.
	
	Second, consider the case that $j_\rb > k-1$. Then $j_\rt > k$. By construction, this implies that $i_\rt \ge t-s+1$. Since $W(P)$ has $t-s+k$ rows in total, $I$ is contained in the last $k$ rows, including the empty $t$-th row. This is again a contradiction.
\end{proof}

The class of non-trivial $Q_1$-like patterns is closed under rotation, so \Cref{p:vert_wit_suff} and \Cref{p:wp_vert_wit} imply \Cref{p:q1-like}.

\restateQOneLike*

\subsection{Some \texorpdfstring{$Q_0$}{Q\_0}-like permutation matrix patterns}

One can manually check that $W(P)$ avoids $P$ even for many $Q_0$-like patterns, such as $Q_5$. We refine \Cref{p:wp-q1-like} to cover more patterns, including $Q_5$ and all but four of the not-once-separable $6 \times 6$ permutation matrices, up to reflection. For three of the remaining patterns, we individually show that $W(P)$ yields a witness. For the last pattern, we construct a witness by modifying $W(P)$ slightly. This shows that every $Q_0$-like permutation matrix of size at most 6 has a vertical witness. Since these patterns are closed under inversion, they all have bounded saturation function. Together with \Cref{p:q1-like}, we obtain:

\restateSmallPatterns*

Let $I = \{(i_\rb, j_\rb), (i_\rt, j_\rt)\}$ be an occurrence of $\Iprimetwomat$ in some matrix, i.e., two 1-entries with $i_\rb > i_\rt$ and $j_\rb < j_\rt$. We define the \emph{height} if $I$ as $i_\rt - i_\rb + 1$, the number of rows containing an entry in $I$ or between the two entries of $I$.

We first consider $Q_0$-like patterns that contain an occurrence of $\Iprimetwomat$ with height $k-1$, which, among others, covers all but four permutation matrices of size at most $6$. Observe that permutation matrices of this type can be though of as \emph{almost} $Q_1$-like: Removing the top (or bottom) row and then the new empty column creates a $Q_1$-like permutation matrix. We first prove some facts about occurrences of $\Iprimetwomat$ in $W(P)$.

\begin{lemma}\label{p:occ-not-height-k}
	Let $P$ be a non-trivial $Q_0$-like $k \times k$ pattern. Then each occurrence $I$ of $\Iprimetwomat$ in $W(P)$ has height at most $k-1$.
\end{lemma}
\begin{proof}
	Suppose there is an occurrence $I = \{(i_\rb, j_\rb), (i_\rt, j_\rt)\}$ of $\Iprimetwomat$ in $W(P)$ of height at least $k$, i.e., $j_\rb < j_\rt$ and $i_\rb - i_\rt \ge k-1$. We claim that $I$ is completely contained in one of the two partial copies of $P$ in $W(P)$, i.e., either $j_\rb < j_\rt \le k-1$ or $k-1 < j_\rb < j_\rt$. This implies that there is also a height-$k$ occurrence of $\Iprimetwomat$ in $P$, which contradicts the assumption that $P$ is $Q_0$-like. It remains to show our claim.
	
	Let $s$ and $t$ be the rows of the leftmost and rightmost 1-entry in $P$. Towards our claim, suppose on the contrary that $j_\rb \le k-1 < j_\rt$. Then $i_\rt \ge (t-s)+1$ by construction, and thus $i_\rb \ge k+(t-s) > k$. But then $(i_\rb, j_\rb)$ cannot be a 1-entry, a contradiction.
\end{proof}

\begin{lemma}\label{p:occ-height-empty-row}
	Let $P$ be a non-trivial $Q_0$-like $k \times k$ pattern. Then each occurrence $I$ of $\Iprimetwomat$ in $W(P)$ with height $k-1$ has the empty row between its two entries.
\end{lemma}
\begin{proof}
	Let $s$ and $t$ be the rows of the leftmost and rightmost 1-entry in $P$, so $W(P)$ is a $(k+t-s) \times (2k-2)$ matrix where the $t$-th row is empty. Since $P$ is $Q_0$-like, we have $s \ge 2$ and $t \le k-1$. Consider an occurrence $I = \{(i_\rb, j_\rb), (i_\rt, j_\rt)\}$ of $\Iprimetwomat$ in $W(P)$ where $i_\rb - i_\rt = k-2$. We have
	\begin{align*}
		& i_\rt = i_\rb - k+2 \le  k + t - s - k + 2 = t-s+2 \le t \text{, and}\\
		& i_\rb = k-2 + i_\rt \ge k-1 \ge t.
	\end{align*}
	Since the $t$-th row is empty, we also have $i_\rt \neq t \neq i_\rb$, and thus $i_\rt < t < i_\rb$.
\end{proof}

\begin{proposition}\label{p:almost-q1}
	Let $P$ be a non-trivial $Q_0$-like $k \times k$ pattern that contains an occurrence $\Iprimetwomat$ of height $k-1$. Then $W(P)$ avoids $P$.
\end{proposition}
\begin{proof}
	Suppose that $P$ is contained in $W(P)$. Then $W(P)$ must contain an occurrence $I$ of $\Iprimetwomat$, such that there are $k-3$ non-empty rows between the two entries in $I$. This means that either $I$ has height at least $k$, or $I$ has height $k-1$ and there are no empty rows between its two entries. The former is impossible by \Cref{p:occ-not-height-k}, the latter is impossible by \Cref{p:occ-height-empty-row}.
\end{proof}

There are four remaining not-once-separable $Q_0$-like permutation matrices of size at most 6. \Cref{fig:medium-q0-pats} shows them along with vertical witnesses.

\begin{figure}
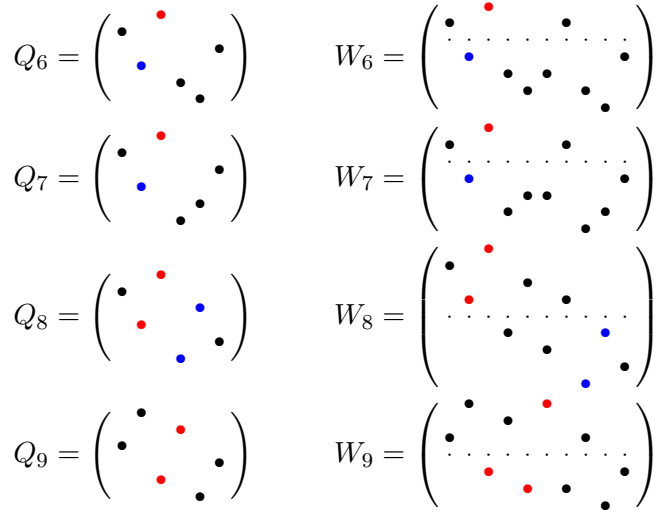

	\centering
	\begin{align*}
		Q_6 =
		\begin{smallbulletmatrix}
			  &  &\r&  &  &  \\
			\o&  &  &  &  &  \\
			  &  &  &  &  &\o\\
			  &\b&  &  &  &  \\
			  &  &  &\o&  &  \\
			  &  &  &  &\o&  
		\end{smallbulletmatrix}
		\hspace{10mm} W_6 =
		\begin{smallbulletmatrix}
			  &  &\r&  &  &  &  &  &  &  \\
			\o&  &  &  &  &  &\o&  &  &  \\
			\d&\d&\d&\d&\d&\d&\d&\d&\d&\d\\
			  &\b&  &  &  &  &  &  &  &\o\\
			  &  &  &\o&  &\o&  &  &  &  \\
			  &  &  &  &\o&  &  &\o&  &  \\
			  &  &  &  &  &  &  &  &\o&  
		\end{smallbulletmatrix}
		\\
		Q_7 =
		\begin{smallbulletmatrix}
			  &  &\r&  &  &  \\
			\o&  &  &  &  &  \\
			  &  &  &  &  &\o\\
			  &\b&  &  &  &  \\
			  &  &  &  &\o&  \\
			  &  &  &\o&  &  
		\end{smallbulletmatrix}
		\hspace{10mm} W_7 =
		\begin{smallbulletmatrix}
			  &  &\r&  &  &  &  &  &  &  \\
			\o&  &  &  &  &  &\o&  &  &  \\
			\d&\d&\d&\d&\d&\d&\d&\d&\d&\d\\
			  &\b&  &  &  &  &  &  &  &\o\\
			  &  &  &  &\o&\o&  &  &  &  \\
			  &  &  &\o&  &  &  &  &\o&  \\
			  &  &  &  &  &  &  &\o&  &  
		\end{smallbulletmatrix}
		\\
		Q_8 =
		\begin{smallbulletmatrix}
			  &  &\r&  &  &  \\
			\o&  &  &  &  &  \\
			  &  &  &  &\b&  \\
			  &\r&  &  &  &  \\
			  &  &  &  &  &\o\\
			  &  &  &\b&  &  
		\end{smallbulletmatrix}
		\hspace{10mm} W_8 =
		\begin{smallbulletmatrix}
			  &  &\r&  &  &  &  &  &  &  \\
			\o&  &  &  &  &  &  &  &  &  \\
			  &  &  &  &\o&  &  &  &  &  \\
			  &\r&  &  &  &  &\o&  &  &  \\
			\d&\d&\d&\d&\d&\d&\d&\d&\d&\d\\
			  &  &  &\o&  &  &  &  &\b&  \\
			  &  &  &  &  &\o&  &  &  &  \\
			  &  &  &  &  &  &  &  &  &\o\\
			  &  &  &  &  &  &  &\b&  &  
		\end{smallbulletmatrix}
		\\
		Q_9 =
		\begin{smallbulletmatrix}
			  &\o&  &  &  &  \\
			  &  &  &\r&  &  \\
			\o&  &  &  &  &  \\
			  &  &  &  &  &\o\\
			  &  &\r&  &  &  \\
			  &  &  &  &\o&  
		\end{smallbulletmatrix}
		\hspace{10mm} W_9 =
		\begin{smallbulletmatrix}
			  &\o&  &  &  &\r&  &  &  &  \\
			  &  &  &\o&  &  &  &  &  &  \\
			\o&  &  &  &  &  &  &\o&  &  \\
			\d&\d&\d&\d&\d&\d&\d&\d&\d&\d\\
			  &  &\r&  &  &  &  &  &  &\o\\
			  &  &  &  &\r&  &\o&  &  &  \\
			  &  &  &  &  &  &  &  &\o&  
		\end{smallbulletmatrix}
	\end{align*}
	\caption{Remaining $Q_0$-like $6 \times 6$ permutation matrices with vertical witnesses.}\label{fig:medium-q0-pats}
\end{figure}

\begin{proposition}\label{p:remaining-four-q0}
	For $i \in \{6,7,8,9\}$, the matrix $W_i$ is a vertical witness for $Q_i$.
\end{proposition}
\begin{proof}
	For $i \in \{6,7,8\}$, we simply have $W_i = W(P_i)$. Thus, it suffices to show that $W_i$ avoids $Q_i$. For $i \in \{6,7\}$, note that $W_i$ has height $k+1$, and therefore only $k$ non-empty rows. As such, an occurrence of $Q_i$ in $W_i$ must map the topmost 1-entry of $Q_i$ to the topmost 1-entry of $W_i$ (marked red in the figure). It is easy to see that then the 1-entry in the second column of $Q_i$ must be mapped to the second column of $W_i$ (marked blue). But then the 1-entries in the second and third column of $Q_1$ cannot be correctly mapped to 1-entries in $W_i$.
	
	Considering $i = 8$, observe that $Q_8$ has two occurrences $I_1, I_2$ of $\Iprimetwomat$ of height $k-2 = 4$ (marked red/blue in the figure). All occurrences of $\Iprimetwomat$ in $W_8$ have height at most 4, and out of the occurrences of height 4, only two, say $I_1'$ and $I_2'$ (marked red/blue), do not contain the empty row. Thus, an occurrence of $Q_8$ in $W_8$ must map $I_1, I_2$ to $I_1', I_2'$. However, $I_1, I_2$ span overlapping rows, while $I_1', I_2'$ do not, a contradiction.
	
	Finally, consider $i = 9$. The matrix $W_9$ is almost equal to $W(Q_9)$; the only difference is that the entry in the 6-th column is moved one row up. Note that this entry is the highest 1-entry in the right partial copy of $Q_9$ in $W(Q_9)$. Since we only move the highest entry up, the right partial copy stays intact in some sense. In particular, adding a 1-entry in the left half of the $t$-th row will still complete an occurrence of $Q_9$. The same is true for the right half of the $t$-th row, since the left partial copy is not changed. Thus, the $t$-th row of $W_9$ is expandable.
	
	We still have to argue that $W_9$ does not contain $Q_9$. Suppose otherwise. Observe that $Q_9$ contains exactly one occurrence $I$ of $\Iprimetwomat$ of height 4 (marked in red in the figure). All occurrences of $\Iprimetwomat$ in $W_9$ of height at least 4 have the empty row in between their two entries, so $I$ must be mapped to the some occurrence $I'$ of $\Iprimetwomat$ in $W_9$ of height larger than 4. There are only two such occurrences (marked in red), both involving the entry in the sixth column of $W_9$. However, the top entry in $I'$ is in the first row of $W_9$, but the top entry of $I$ is in the second row of $Q_9$, leaving no room for the top entry of $Q_9$. This means $Q_9$ is not contained in $W_9$.
\end{proof}

\Cref{p:almost-q1,p:remaining-four-q0} show that each not-once-separable $Q_0$-like permutation matrix of size at most 6 has a vertical witness. As discussed at the start of this section, this implies \Cref{p:small-patterns}. For convenience, we list all not-once-separable $Q_0$-like permutation matrices of size at most 6 in \Cref{app:small-perms}.

\section{Conclusion}

Fulek and Keszegh \cite{FulekKeszegh2020} showed that the saturation function of once-separable patterns is linear. We extend their result by showing that all non-trivial $Q_1$-like patterns have bounded saturation function. In particular, this is another step towards the classification of permutation matrices, leaving only the $Q_0$-like permutation matrices. We find many more $Q_0$-like permutation matrices with bounded saturation function. This completes the classification of permutation matrices of size at most 6, showing that a permutation matrix of size at most 6 has linear saturation function if and only if it is once-separable. It seems possible that this is true for all permutation matrices.

\begin{openQuestion}
	Is the saturation function bounded for each not-once-separable permutation matrix?
\end{openQuestion}

Our witness construction $W(P)$ undoubtedly works for a larger class of matrices than we identified (cf.\ \Cref{p:remaining-four-q0}). However, we also provide an example of a not-once-separable $Q_0$-like permutation matrix ($Q_9$) for which our construction does \emph{not} yield a vertical witness. It would be interesting to precisely identify the patterns where the construction works.

\begin{openQuestion}
	Is there a simple characterization of patterns $P$ where $W(P)$ avoids~$P$?
\end{openQuestion}

Our results also extend to certain non-permutation matrices, but we did not consider matrices with empty rows or columns or with more than one 1-entry in either of the first or last row or column. We note, however, that \Cref{p:vert_wit_suff} still may be useful for patterns that have only one 1-entry in the the last row and only one 1-entry in the last column, but multiple 1-entries in the first row and column.

\pagebreak

\bibliography{info}{}
\bibliographystyle{alpha}

\newpage
\appendix

\section{Small permutation matrices}\label{app:small-perms}

The following table lists all not-once-separable and $Q_0$-like permutation matrices of size at most $6 \times 6$, up to reflection. For each matrix, we reference the proof that it has bounded saturation function. Whenever \Cref{p:almost-q1} is used, the relevant occurrence of $\Iprimetwomat$ is highlighted in red.

\bigskip

\begin{minipage}[c]{0.5\textwidth}
	\centering
	\begin{TAB}[4pt]{|c|c|}{|c|c|c|c|c|c|c|c|c|}
		$\begin{smallbulletmatrix}
		  &  &\r&  &  \\
		\o&  &  &  &  \\
		  &  &  &  &\o\\
		  &\r&  &  &  \\
		  &  &  &\o&  
		\end{smallbulletmatrix}$ & \Cref{p:almost-q1} ($Q_5$)\\
		$\begin{smallbulletmatrix}
		  &  &  &\r&  &  \\
		\o&  &  &  &  &  \\
		  &\o&  &  &  &  \\
		  &  &  &  &  &\o\\
		  &  &\r&  &  &  \\
		  &  &  &  &\o&  
		\end{smallbulletmatrix}$ & \Cref{p:almost-q1}\\
		$\begin{smallbulletmatrix}
		  &  &\o&  &  &  \\
		\o&  &  &  &  &  \\
		  &  &  &  &  &\o\\
		  &\o&  &  &  &  \\
		  &  &  &\o&  &  \\
		  &  &  &  &\o&  
		\end{smallbulletmatrix}$ & \Cref{p:remaining-four-q0} ($Q_6$)\\
		$\begin{smallbulletmatrix}
		  &  &\o&  &  &  \\
		\o&  &  &  &  &  \\
		  &  &  &  &\o&  \\
		  &\o&  &  &  &  \\
		  &  &  &  &  &\o\\
		  &  &  &\o&  &  
		\end{smallbulletmatrix}$ & \Cref{p:remaining-four-q0} ($Q_8$)\\
		$\begin{smallbulletmatrix}
		  &  &\o&  &  &  \\
		\o&  &  &  &  &  \\
		  &  &  &  &  &\o\\
		  &\o&  &  &  &  \\
		  &  &  &  &\o&  \\
		  &  &  &\o&  &  
		\end{smallbulletmatrix}$ & \Cref{p:remaining-four-q0} ($Q_7$)\\
		$\begin{smallbulletmatrix}
		  &  &  &\r&  &  \\
		\o&  &  &  &  &  \\
		  &  &  &  &  &\o\\
		  &\o&  &  &  &  \\
		  &  &\r&  &  &  \\
		  &  &  &  &\o&  
		\end{smallbulletmatrix}$ & \Cref{p:almost-q1}\\
		$\begin{smallbulletmatrix}
		  &  &\r&  &  &  \\
		\o&  &  &  &  &  \\
		  &  &  &\o&  &  \\
		  &  &  &  &  &\o\\
		  &\r&  &  &  &  \\
		  &  &  &  &\o&  
		\end{smallbulletmatrix}$ & \Cref{p:almost-q1}\\
		$\begin{smallbulletmatrix}
		  &  &\r&  &  &  \\
		\o&  &  &  &  &  \\
		  &  &  &  &  &\o\\
		  &  &  &\o&  &  \\
		  &\r&  &  &  &  \\
		  &  &  &  &\o&  
		\end{smallbulletmatrix}$ & \Cref{p:almost-q1}\\
		$\begin{smallbulletmatrix}
		  &  &\r&  &  &  \\
		\o&  &  &  &  &  \\
		  &  &  &  &\o&  \\
		  &  &  &  &  &\o\\
		  &\r&  &  &  &  \\
		  &  &  &\o&  &  
		\end{smallbulletmatrix}$ & \Cref{p:almost-q1}
	\end{TAB}
\end{minipage}
\begin{minipage}[c]{0.3\textwidth}
\centering
\begin{TAB}[4pt]{|c|c|}{|c|c|c|c|c|c|c|c|}
	$\begin{smallbulletmatrix}
		  &  &\r&  &  &  \\
		\o&  &  &  &  &  \\
		  &  &  &  &  &\o\\
		  &  &  &  &\o&  \\
		  &\r&  &  &  &  \\
		  &  &  &\o&  &  
	\end{smallbulletmatrix}$ & \Cref{p:almost-q1}\\
	$\begin{smallbulletmatrix}
		  &  &  &\r&  &  \\
		\o&  &  &  &  &  \\
		  &  &\o&  &  &  \\
		  &  &  &  &  &\o\\
		  &\r&  &  &  &  \\
		  &  &  &  &\o&  
	\end{smallbulletmatrix}$ & \Cref{p:almost-q1}\\
	$\begin{smallbulletmatrix}
		  &  &  &\r&  &  \\
		\o&  &  &  &  &  \\
		  &  &  &  &  &\o\\
		  &  &\o&  &  &  \\
		  &\r&  &  &  &  \\
		  &  &  &  &\o&  
	\end{smallbulletmatrix}$ & \Cref{p:almost-q1}\\
	$\begin{smallbulletmatrix}
		  &\o&  &  &  &  \\
		  &  &  &\o&  &  \\
		\o&  &  &  &  &  \\
		  &  &  &  &  &\o\\
		  &  &\o&  &  &  \\
		  &  &  &  &\o&  
	\end{smallbulletmatrix}$ & \Cref{p:remaining-four-q0} ($Q_9$)\\
	$\begin{smallbulletmatrix}
		  &\o&  &  &  &  \\
		  &  &  &  &\r&  \\
		\o&  &  &  &  &  \\
		  &  &  &  &  &\o\\
		  &  &\o&  &  &  \\
		  &  &  &\r&  &  
	\end{smallbulletmatrix}$ & \Cref{p:almost-q1}\\
	$\begin{smallbulletmatrix}
		  &\o&  &  &  &  \\
		  &  &  &\r&  &  \\
		\o&  &  &  &  &  \\
		  &  &  &  &  &\o\\
		  &  &  &  &\o&  \\
		  &  &\r&  &  &  
	\end{smallbulletmatrix}$ & \Cref{p:almost-q1}\\
	$\begin{smallbulletmatrix}
		  &\o&  &  &  &  \\
		  &  &  &  &\r&  \\
		\o&  &  &  &  &  \\
		  &  &  &  &  &\o\\
		  &  &  &\o&  &  \\
		  &  &\r&  &  &  
	\end{smallbulletmatrix}$ & \Cref{p:almost-q1}\\
	$\begin{smallbulletmatrix}
		  &  &\r&  &  &  \\
		  &  &  &  &\o&  \\
		\o&  &  &  &  &  \\
		  &  &  &  &  &\o\\
		  &\r&  &  &  &  \\
		  &  &  &\o&  &  
	\end{smallbulletmatrix}$ & \Cref{p:almost-q1}
\end{TAB}
\end{minipage}

\end{document}